\documentclass[a4paper,12pt]{article}
\usepackage{amsmath, amssymb, amsthm, amsfonts, stmaryrd}
\usepackage[utf8]{inputenc}
\usepackage[T1]{fontenc}
\usepackage[english,french]{babel}
\usepackage{mathrsfs,dsfont,mathtools,lipsum}

\usepackage[todonotes={textsize=scriptsize}]{changes}
\usepackage[colorlinks=true,linkcolor=blue, citecolor=red]{hyperref}

\usepackage[marginparwidth=2.5cm]{geometry}

\usepackage[active]{srcltx}
\usepackage{color}

\DeclareFontFamily{U}{mathx}{\hyphenchar\font45}
\DeclareFontShape{U}{mathx}{m}{n}{
      <5> <6> <7> <8> <9> <10>
      <10.95> <12> <14.4> <17.28> <20.74> <24.88>
      mathx10
      }{}
\DeclareSymbolFont{mathx}{U}{mathx}{m}{n}
\DeclareFontSubstitution{U}{mathx}{m}{n}
\DeclareMathAccent{\widecheck}{0}{mathx}{"71}

\DeclareMathOperator*{\esssup}{ess\,sup}

\newcommand{\dd}{\mathrm{d}}
\newcommand{\Uin}{U_{\mathrm{in}}}
\newcommand{\uin}{u_{\mathrm{in}}}
\newcommand{\vin}{v_{\mathrm{in}}}
\newcommand{\T}{\mathbf{T}}
\newcommand{\ffi}{\varphi}

\newcommand{\Ll}{\textnormal{L}_{\textnormal{loc}}}

\def\Z{\mathbf{Z}}
\def\N{\mathbf{N}}
\def\R{\mathbf{R}}

\def\C{\textnormal{C}}
\def\L{\textnormal{L}}

\def\D{\textnormal{D}}

\def\B{\textnormal{B}}

\def\W{\textnormal{W}}

\def\E{\textnormal{E}}
\def\H{\textnormal{H}}
\def\V{\textnormal{V}}
\def\pa{\partial}

\newtheorem{thm}{Theorem}
\newtheorem{cor}{Corollary}
\newtheorem{lem}{Lemma}
\newtheorem{prop}{Proposition}
\newtheorem{definition}{Definition}

\newtheorem{remark}{Remark}

\title{Perturbative global solutions of a large class of cross diffusion systems in any dimension} 

\begin{document}
\selectlanguage{english}

\maketitle
 
\centerline{\scshape L. Desvillettes}
\medskip
{\footnotesize
  \centerline{Universit\'e Paris Cit\'e and Sorbonne Université, CNRS, IUF, IMJ-PRG}
  \centerline{F-75013 Paris, France}
\centerline{E-mail: desvillettes@imj-prg.fr}}
\bigskip

\centerline{\scshape A. Moussa}
\medskip
{\footnotesize
  \centerline{Sorbonne Université, CNRS, Université Paris Cité, Laboratoire Jacques-Louis Lions (LJLL)}
  \centerline{Département de Mathématiques et Applications (DMA), École Normale Supérieure (ENS-PSL)}
\centerline{F-75005, Paris, France}
\centerline{E-mail : ayman.moussa@sorbonne-universite.fr}}
\bigskip

\begin{abstract}
This article focuses on a large  family of cross-diffusion systems of the form $\pa_t U - \Delta A(U) = 0$, in dimension $d \in \N^*$, and where $U \in \R^2$. We show that under natural conditions on the nonlinearity $A$, those systems have a unique smooth (nonnegative for all components) solution when the initial data are small enough in a suitable norm.
\end{abstract}

\section{Introduction}

Cross diffusion systems constitute a large class of systems of PDEs which naturally appear in physics as well as population dynamics (see for instance \cite[Chapter 4]{book_ansgar} for several examples). The term cross diffusion is used when in a system of parabolic PDEs, at least one  of the diffusion rates in the equation for one component depends on the value of one or more of the other components.
\medskip

We focus here on the case when the system can be written under the form 
\begin{equation}
  \label{eq:syst_gen}
  \left\{
    \begin{lgathered}
\partial_t u - \Delta\big[(d_1 + p(u,v))u\big] = 0,\\
\partial_t v - \Delta\big[(d_2 + q(u,v))v\big] = 0,
    \end{lgathered}
  \right.
\end{equation}
where $d_1,d_2>0$ are two constants, and  $p,q$ are  two elements of $\R_+[X,Y]$ (polynomials of two variables with real nonnegative coefficients) vanishing at $(0,0)$, and we consider periodic boundary conditions, that is the system is considered on the torus $\T^d$. It can be checked, at least formally, that a solution of these systems remain nonnegative for all $t\geq 0$ whenever it is initially nonnegative. In the sequel of this paper, we systematically denote $U := (u,v)$.

\medskip

Note that the classical Shigesada-Kawasaki-Teramoto (called SKT in the rest of the paper) system (see \cite{skt}), where $p$ and $q$ are linear w.r.t $u,v$, enters in this category (if the zero-th order  competition terms are dropped). Many natural extensions of this system where $p,q$ are not linear any more also enter this category.  
\medskip

Before stating our main theorems, let us fix a (rather weak) setting for which each equation of \eqref{eq:syst_gen} is well-posed. For this purpose, we rely on the following result.

\begin{lem}\label{lem:dua:bis}
  Consider $\mu\in\Ll^\infty(\R_+;\L^\infty(\T^d))$ such that $\inf \mu>0$, $z_{\textnormal{in}}\in\H^{-1}(\T^d)$,
  and $f\in\Ll^2(\R_+;\L^2(\T^d))$. Then, there exists a unique $z\in\Ll^2(\R_+;\L^2(\T^d))$ such that, after extending all functions  by $0$ on $\R_-\times\T^d$, the following holds ($\delta_0$ is the Dirac mass at point $0$ on $\R$)
\begin{align*}
    \partial_t z -\Delta(\mu z) = \Delta f +  \delta_0\otimes z_{\textnormal{in}}, \text{ in }\mathscr{D}'(\R\times\T^d).
\end{align*}
Furthermore, this solution $z$ belongs to $\mathscr{C}([0,T];\H^{-1}(\T^d))$ and satisfies the following estimate for $t\in[0,T]$:
    \begin{multline}\label{ineq:dua:bisbis}
\|z(t, \cdot)\|_{\H^{-1}(\T^d)}^2 +\int_0^t \int_{\T^d} \mu\, z^2 \\\leq \|z_{in}\|_{\H^{-1}(\T^d)}^2 + \left(\int_{\T^d}z_{in}\right)^2 \int_0^t\int_{\T^d} \mu + \int_0^t \int_{\T^d} \frac{f^2}{\mu}.
    \end{multline}
  \end{lem}
  \begin{proof}
See \cite[Lemma 1]{bmmh} (for the estimate) and \cite[Theorem 3]{m} (for the well-posedness).
   \end{proof}
   It may be readily checked that for nonnegative bounded densities $u,v$, each equation of \eqref{eq:syst_gen} is a particular case of the Kolmogorov equation considered in Lemma~\ref{lem:dua:bis}, in the case of a vanishing source term $f=0$. Estimate \eqref{ineq:dua:bisbis} is a generalization of the celebrated \emph{duality estimate} (see \cite{pisc} for the seminal version and \cite{perthame_laamri} for a similar one). The case $f\neq 0$ will be used crucially in our stability result Theorem~\ref{uniqueness} below, generalizing the analysis done in \cite{bmmh} for the classical SKT system.

   \begin{definition}\label{def:dua}
In the well-posedness setting of Lemma~\ref{lem:dua:bis}, we speak of \textsf{dual solution}. We  use the same naming for solutions $(u,v)$ of \eqref{eq:syst_gen} that match this setting for each equation. 
   \end{definition}

   \medskip
   
   We now write down the main results of this paper (Theorems \ref{existence} and \ref{uniqueness}), respectively related to existence and stability/uniqueness of solutions to system (\ref{eq:syst_gen}). We need for that to define the following norm.
\begin{definition}\label{def:Nk}
For $k\in[1,\infty]$, and $\ffi \in \mathscr{C}^\infty(\T^d)$, let
      \begin{align*}
N_k(\ffi) := \left(\int_0^{+\infty} \|e^{t\Delta}\ffi\|_k^k\,\dd t\right)^{1/k}.
      \end{align*}
  It is clear thanks to Minkowski's inequality that $N_k$ is a norm,  and the corresponding completion of $\mathscr{C}^\infty(\T^d)$
 will be denoted $\B_k(\T^d)$ in the sequel.
    \end{definition}
    \begin{remark}
This choice of notation $\B_k(\T^d)$ for the completion is here to remind the reader that such a norm is common in the study of Besov spaces. Actually, this norm is equivalent to the one which usually defines $\dot{\B}_{k,k}^{-2/k}(\T^d)$ (see for instance \cite[Theorem 2.34]{bachra} for the euclidean setting). It can be directly checked here that in the case when $k=2$, the norm $N_k$ is equivalent to the $\H^{-1}(\T^d)$ norm.
 \end{remark}

 \begin{thm}\label{existence} Let $d\in \N^*$,  $d_1,d_2>0$ be two constants, and $p,q$ be  two elements of $\R_+[X,Y]$ vanishing at $(0,0)$. We also consider $k > 1 + \frac{d}2$.
\par
Then there exists a constant $\delta>0$ (depending on $d$, $d_1,d_2, p, q$ and $k$) such that for nonnegative initial data $\uin, \vin  \in \L^\infty(\T^d)\cap \H^{-1}(\T^d)$ satisfying 
$$ \|\uin \|_{\infty} + \|\vin \|_{\infty}   + N_k( \Delta \{ [d_1 + p(\uin,\vin)]\,\uin \}) + N_k(\Delta \{ [d_2 + q(\uin,\vin)]\,\vin \})  \leq \delta/2 , $$
there exists a global, bounded, nonnegative for each component, $(u,v) \in \mathscr{C}^0(\R_+;\H^{-1}(\T^d))$,
 dual solution to system \eqref{eq:syst_gen}, in the sense of Definition~\ref{def:dua}. This solution satisfies furthemore
\begin{align*}
\esssup_{t\geq 0} \bigg( \|u(t)\|_\infty +\|v(t)\|_\infty  \bigg) \leq \delta,
  \end{align*}
  and
  \begin{align}\label{ineq:large}
\left\|u(t)-\int_{\T^d}\uin \right\|_2 + \left\|v(t)-\int_{\T^d}\vin \right\|_2 \,\operatorname*{\longrightarrow}_{t\rightarrow +\infty} \,0,
  \end{align}
  with an exponential rate.
 \end{thm}

 \begin{thm}\label{uniqueness} Let $d\in \N^*$, 
 $d_1,d_2>0$ be two constants, and $p,q$ be  two elements of $\R_+[X,Y]$ vanishing at $(0,0)$. Then there exist two constants $\delta, \textnormal{C}>0$ depending on $d$, $d_1,d_2, p, q$ such that (when $T>0$) for any pair $(u_1,v_1)$ and $(u_2,v_2)$ of $\L^\infty([0,T];\L^\infty(\T^d))\cap\mathscr{C}^0([0,T];\H^{-1}(\T^d))$ nonnegative dual solutions to system (\ref{eq:syst_gen}) in the sense of Definition~\ref{def:dua}  which are such that (for $i=1,2$)
   \[ \esssup_{t\in[0,T]} \bigg( \|u_i(t)\|_{\infty} + \|v_{i}(t) \|_{\infty} \bigg) \le \delta,\]
 the following stability estimate for $t\in[0,T]$ holds: 
\begin{multline*} \| (u_1 - u_2)(t)||_{\H^{-1}(\T^d)} + \|(v_1 - v_2)(t)||_{\H^{-1}(\T^d)}\\ \le  ||(u_1 - u_2)(0) ||_{\H^{-1}(\T^d)} + ||(v_1 - v_2)(0) ||_{\H^{-1}(\T^d)}\\
  +\, \textnormal{C}\,T\, \bigg( \bigg|\int_{\T^d} (u_1 - u_2)(0) \bigg| +    \bigg|\int_{\T^d} (v_1 - v_2)(0) \bigg|  \bigg).
  \end{multline*}
\end{thm}

\begin{remark}
We have chosen a simple formulation for Theorem \ref{uniqueness}, in which both solutions are small in $\L^{\infty}$, for the sake of clarity. However, our method actually proves a finer statement in which only one of the two solutions has to be small (the other one being given), see Proposition~\ref{prop:uni}.
\end{remark}
Combining Theorem~\ref{uniqueness} with Theorem~\ref{existence}, we infer the following corollary of well-posedness for the system.
\begin{cor}
Under the assumptions of Theorem~\ref{existence}, if $\delta$ is chosen sufficiently small, the solutions given in Theorem~\ref{existence}  satisfy the stability estimate of Theorem~\ref{uniqueness}. In particular, such solutions are unique (among those satisfying the assumptions required in Theorem~\ref{uniqueness}).  
\end{cor}

SKT-like systems have been studied for a long time. As they were originally introduced in \cite{skt} for the emergence of patterns formation, the first articles focused on their stationary solutions, see \cite{mami,mim,shi} for instance. Existence of solutions for these systems has been the source of a large literature which interestingly enough only started 5 years after the publication of \cite{skt}. Indeed, the first existence result seems to be the one of Kim in \cite{kim1984smooth}, which established that there are global (strong) solutions but only for equal diffusion rates, in dimension $1$. After several partial results of this kind, an important step was made by Amann \cite{amann90,amann90b}, paving the way for a systematic local existence theory (with a blowup criterion) under a mild elliptic condition on the system (also known as \og normal ellipticity\fg{}, that is Petrovskii's condition for parabolic systems). In the wake of Amann's theory, a substantial number of articles attempted to prove the existence of global strong solutions of SKT-like systems through a uniform control of Amann's blowup criterion, generally at the price of restrictive constraints on the system. For instance, in the so-called triangular case, that is when cross diffusion is present only in one of the two equations of the system, Lou, Ni and Wu obtained existence and regularity in dimension $2$ in \cite{lnw}, whereas later Choi, Lui and Yamada first got rid of the restriction on the dimension in \cite{cly1}, at the price of considering small cross diffusion rates. More recently, \cite{HoanNguPha} established the existence of global smooth solutions for triangular SKT-like systems when some self diffusion is present,  and
\cite{guementre} proved the same kind of results for natural extensions of triangular SKT-like systems.

\medskip

In the general (non triangular) case, global weak solutions of the SKT system were proven to exist in full generality in \cite{CJ} ; a few years later \cite{DLM,DLMT}, similar global solutions were built for some systems having a structure like \eqref{eq:syst_gen}, with possibly non-polynomial diffusion rates (see also \cite{LM,chen_daus_jungel,chenjunwan} for multi-species cases). One common aspect of all these works is that one assumes the existence of an \emph{entropy structure} for the studied system, which gives rise to a convex Lyapunov functional. Its dissipation  allows a control of the gradients of the solutions. 

\medskip

Be it for local, global, weak or strong solutions, in all the previous results, the considered systems always satisfy the normal ellipticity condition. Articles dealing with parabolic systems potentially failing to satisfy this condition are less frequent in the literature  (see \cite{druet_hopf_jungel,xu} for recent examples). In a perturbative setting, it was first noticed in \cite{BLMP} (in dimension $1$), that one can obtain global strong solutions, when the initial data are small. There, the very same energy estimate that one wishes to obtain thanks to some uniform bound can precisely be used to prove those uniform bounds ; this is due to the embedding $\H^1(\T)\hookrightarrow\L^\infty(\T)$ and it allows to prove Theorem 2.1 in \cite{BLMP} that we reproduce here in the specific case of (conservative) system \eqref{eq:syst_gen} for the sake of clarity (we use the notation $\lesssim$ for ``less than a constant times'' here and in the rest of the paper).

\begin{thm}\label{thmo} (consequence of  \cite[Theorem 2.1]{BLMP})
  In dimension $d=1$, there exists $\delta >0$ such that, for any smooth initial data $\Uin:=(\uin,\vin)$ satisfying $\|\Uin\|_\infty + \|\nabla \Uin\|_2\leq \delta$, any smooth solution $U:=(u,v)$ to system \eqref{eq:syst_gen} defined on $\R_+ \times \T^d$ (initialized by $U_{in}$) satisfies
  \begin{align*}
\sup_{t\geq 0}\|U(t)\|_\infty \lesssim \delta,\qquad \sup_{t\geq 0} \|\nabla U(t)\|_2 \lesssim \delta.
    \end{align*}
  \end{thm} 

  \begin{remark}
This estimate allows to build global uniformly bounded solutions for all times, after introducing an adequate approximate scheme. 
\end{remark}

Note that our Theorem \ref{existence} can be seen as an extension of Theorem \ref{thmo} in the case of dimension $d \ge 1$, since it is possible to take $k=2$ when $d=1$, so that the norms used in the two theorems become equivalent.
\medskip

We explain here how our result could probably be generalized. 
\medskip

\begin{itemize}
\item
First, the fact that $p,q$ are polynomials vanishing at point $(0,0)$ with nonnegative coefficients is not essential to our analysis (though it helps greatly in the presentation of the results). It is indeed sufficient to suppose that $p,q$ are nonnegative functions of class $\mathscr{C}^1((\R_+)^2)$ which vanish at point $(0,0)$ (note that the explicit estimate of stability has then to be slightly changed). 
\medskip

\item
Secondly, we think that our analysis could be reproduced, without much change, when systems of more than two equations are considered 
(in this direction, see for instance \cite{chenjunwan} for the existence of global weak solutions for multi-species system).
\medskip

\item
Finally, we insist that our methods cannot work (at least when one looks for global in time solutions) if reaction terms are introduced in the cross diffusion systems, except in very specific cases. For example, if one considers competition terms of Lotka-Volterra type (such as they naturally appear in SKT-type models), one cannot hope to have solutions which remain small when the time becomes large, except if the coefficients appearing in the competition terms are such that associated steady solutions are themselves small. We however think that our result can be extended in at least one general situation, namely when the reaction terms are respectively of the form  $-u\, s_1(u,v)$, $-v\,s_2(u,v)$, where $s_1,s_2 \in \R_+[X,Y]$. Such a form can correspond for example, in terms of modeling, to situations in population dynamics in which no new individuals are introduced (no births).
\end{itemize}

\bigskip

In section \ref{sec:bounduni}, we present the main {\it{a priori}} estimate underlying the proof of Theorem \ref{existence}, while we rigorously prove this Theorem in section \ref{sec3}. Then, the proof of Theorem \ref{uniqueness} is written is section \ref{subsec:locstab}, and some counterexamples are presented in section \ref{sec-cc}. Finally, Appendices A, B and C summarize 
some classical results (mainly from the theory of parabolic PDEs), while Appendix D is dedicated to the presentation of an alternative to the {\it{a priori}} estimate of  section \ref{sec:bounduni}.

\section{Global uniform \emph{a priori} estimates for small initial data}\label{sec:bounduni}

In this section, we write down the \emph{a priori} estimate which enables to show the existence result (that is, Theorem \ref{existence}). We recall Definition~\ref{def:Nk} of the $N_k$ norm, for $k\in[1,\infty]$, and write down the

\begin{prop}\label{prop:uni}
Let $d\in \N^*$,  $d_1,d_2>0$ be two constants, and $p,q$ be  two elements of $\R_+[X,Y]$ vanishing at $(0,0)$. We also consider $k > 1 + \frac{d}2$.
\par
Then there exists a constant $\delta>0$ (depending on $d$, $d_1,d_2, p, q$ and $k$) such that for
any smooth nonegative (for each component) initial data $U_{in}:=(\uin,\vin)$ satisfying
  \begin{multline}\label{ineq:init}
\|\uin\|_{\infty} + \|\vin\|_{\infty} \\+ N_k(\Delta \{[d_1 + p(\uin,\vin)]\,\uin \} ) + N_k( \Delta \{ [d_2 + q(\uin,\vin)]\,\vin \}) 
< \frac{\delta}{2},
\end{multline}
any smooth nonnegative  (for each component) solution $U:=(u,v)$ (defined on $\R_+ \times \T^d$) to  system \eqref{eq:syst_gen}  (initialized by $U_{in} := (\uin, \vin)$), satisfies
  \begin{align*}
\sup_{t\geq 0}\|U(t, \cdot)\|_\infty < \delta.
    \end{align*}
  \end{prop}

  \begin{proof}
We introduce the notations
\begin{align*}
  \Phi_1 &:= \Phi_1(u,v) = (d_1+p(u,v))u,\\
  \Phi_2 &:= \Phi_2(u,v) = (d_2+q(u,v))v,
\end{align*}
together with $\Phi_{i,\mathrm{in}} := \Phi_i(\uin, \vin)$. 
A straightforward computation shows the existence, for $i=1,2$, of polynomials $Q_i,R_i\in\R[X,Y]$ for which, for $i\neq j \in\{1,2\}$,
\begin{align}\label{eq:phi}
  \partial_t \Phi_i - d_i \Delta\Phi_i =  Q_i(u,v) \Delta \Phi_i + R_i(u,v) \Delta \Phi_j.
\end{align}
 
Since the polynomials $p,q,$  vanish at $(0,0)$, one checks that this is also the case of the polynomials $Q_i$ and $R_i$.
\medskip

 From now on, $P$ will denote an element of $\R_+[X]$ satisfying $P(0)=P'(0) = 0$. This polynomial may vary (a finite number of times) from line to line in the coming proof.

\medskip

We introduce, for $T>0$, $Q_T := [0,T] \times \T^d$, and
\begin{align*}
  \lambda(T) := \|\Delta\Phi_1\|_{\L^k(Q_T)} + \|\Delta\Phi_2\|_{\L^k(Q_T)} + \|u\|_{\L^\infty(Q_T)} +\|v\|_{\L^\infty(Q_T)}.
\end{align*}
Also, $\lambda(0) = ||\uin||_{\infty} +  ||\vin||_{\infty} $, so that (since $\uin$ and $\vin$ are bounded)  $ \sum_{i=1,2}\|\Phi_{i,\mathrm{in}}\|_\infty\ \lesssim \lambda(0)).$ 
\par
Thanks to identity \eqref{eq:phi},
\begin{align}\label{ineq:phi}
\|\partial_t \Phi_i-d_i\Delta\Phi_i\|_{\L^k(Q_T)} \leq P(\lambda(T)). 
\end{align}
Since the functions $\Phi_i$ solve the equations \eqref{eq:phi} (with initial data $\Phi_{i,\mathrm{in}}$), we can use Theorem~\ref{lem:unibound} of Appendix C (semigroup property of the heat equation in the torus)
 to  get the bound
\begin{align*}
  \sum_{i=1,2}  \|\Phi_i\|_{\L^\infty(Q_T)} &\lesssim \sum_{i=1,2}\|\Phi_{i,in}\|_\infty +  P(\lambda(T))+\sum_{i=1,2} \sup_{s\in[0,T]} \left|\int_{\T^d} \Phi_i(s)\right| \\
  &\lesssim \lambda(0) + P(\lambda(T)) + \sum_{i=1,2} \sup_{s\in[0,T]} \left|\int_{\T^d} \Phi_i(s)\right|.
\end{align*}
Since $\int_{\T^d} u(t, \cdot)$ is conserved in the evolution of the system,
\begin{align*} 
  \left|\int_{\T^d} \Phi_1(s)\right| &\leq d_1 \int_{\T^d} \uin + \int_{\T^d} p(u(s),v(s))u(s) \\
  &\lesssim \lambda(0) + P(\lambda(T)). 
\end{align*}
We have a similar estimate for the average of $\Phi_2$, and we therefore infer  
  \begin{equation} \label{inr}
\sum_{i=1,2}\|\Phi_i\|_{\L^\infty(Q_T)} \lesssim \lambda(0) + P(\lambda(T)).
  \end{equation}
Since $u$ and $v$ are nonnegative and $p,q\in\R_+[X,Y]$, we know moreover that
\begin{align} \label{nna}
\|u\|_{\L^\infty(Q_T)}+\|v\|_{\L^\infty(Q_T)} \lesssim \|\Phi_1\|_{\L^\infty(Q_T)} + \|\Phi_2\|_{\L^\infty(Q_T)}.
\end{align}
Combining this estimate with  the bound (\ref{inr}),   we  control a part of $\lambda(T)$ (in terms of a nonlinear function of $\lambda(T)$ itself). 
\medskip

Defining now $\widetilde{\Phi}_i := \Phi_i - e^{t\Delta}\Phi_{i,\mathrm{in}}$, we see thanks to (\ref{ineq:phi}) that
\begin{align}\label{ineq:phi_tilde}
\|\partial_t \widetilde{\Phi}_i-d_i\Delta \widetilde{\Phi}_i\|_{\L^k(Q_T)} \leq P(\lambda(T)). 
\end{align}
Applying the maximal regularity estimate of Theorem~\ref{thm:max} of Appendix C to $\widetilde{\Phi}_i$, 
we get the estimate
\begin{align}\label{maxregphiti}
\| \Delta \widetilde{\Phi}_i\|_{\L^k(Q_T)} \leq P(\lambda(T)). 
\end{align}
 Then 
\begin{align*}
\sum_{i=1,2} \|\Delta \Phi_i\|_{\L^k(Q_T)} \lesssim_{k,d} \sum_{i=1,2} \|\Delta \widetilde{\Phi}_i\|_{\L^k(Q_T)}
+ \sum_{i=1,2}  \| \Delta  e^{t\Delta}\Phi_{i,\mathrm{in}} \|_{\L^k(Q_T)}  
\end{align*}
\begin{equation}\label{nnb}
  \lesssim_{k,d}  P(\lambda(T)) + \sum_{i=1,2} N_k(\Delta\Phi_{i,\mathrm{in}}),
\end{equation}
which gives a control on the last part of $\lambda(T)$. 
\medskip

All in all, using estimates (\ref{inr}), (\ref{nna}) and (\ref{nnb}), we therefore established the estimate
\begin{align*}
\lambda(T) \lesssim  \lambda(0)+\sum_{i=1,2} N_k(\Delta\Phi_{i,in})+P(\lambda(T)),
\end{align*}
and the conclusion follows then directly from Lemma~\ref{lem:boot} of Appendix B.
\medskip

 $\qedhere$
\end{proof}

\section{Proof of the existence theorem}\label{sec3}

\medskip


In this section, we propose the

 \begin{proof}[Proof of Theorem \ref{existence}]

\medskip
We introduce 
the following regularized version of system (\ref{eq:syst_gen}): For, $\eta, \delta>0$,
we first smooth out the initial data into nonnegative functions $u_{\textnormal{in},\eta}$ and $v_{\textnormal{in},\eta}$
so that, for all $p\in[1,\infty]$, we have $\|u_{\textnormal{in},\eta}\|_p\leq \|\uin\|_p$ and $\|v_{\textnormal{in},\eta}\|_p\leq \|\vin\|_p$,
  and consider (for $\star$ convolution defined on $\T^d$)
\begin{equation}
  \label{eq:syst_gen_reg}
  \left\{
    \begin{lgathered}
\partial_t u-\Delta\bigg[(d_1+ \rho_\eta \star [p_\delta(u,v)] )u \bigg]=0,\\
      \partial_t v-\Delta\bigg[(d_2+ \rho_\eta \star [q_\delta(u,v)] )v \bigg]=0,                                                                                                                                              \end{lgathered}
  \right.
\end{equation}
    where $\rho_\eta$ is a smooth nonnegative kernel of mass $1$, and 
$p_\delta(u,v) :=p(\min(u,\delta), \min(v,\delta))$, $q_\delta(u,v) :=q(\min(u,\delta), \min(v,\delta))$.
 Letting $a_{\delta} := \|\Delta\rho_\eta\|_\infty \max(p(\delta,\delta), q(\delta, \delta))$, we define the (integrable on $\R_+$)  weight $\gamma_\delta(t) := e^{-2a_\delta t}$, and plan to set up a fixed point procedure on the space $\E:=\L^1_{\gamma_\delta}(\R_+;\L^1(\T^d))$. 
\medskip

More precisely, for $(u,v)\in \E\times\E$, we consider $\widetilde{u},\widetilde{v}$ the two solutions of the Kolmogorov equations
\begin{align*}                                                                                                                                                                                                         \partial_t \widetilde{u}-\Delta\bigg[(d_1+ \rho_\eta \star [p_\delta(u^+,v^+)] )\widetilde{u} \bigg]&=0,\\
      \partial_t \widetilde{v}-\Delta\bigg[(d_2+ \rho_\eta \star [q_\delta(u^+,v^+)] )\widetilde{v} \bigg]&=0,                                                                                                                                          \end{align*}
    initialized by $u_{in,\eta}$ and $v_{in,\eta}$ (and where we use the notation $w^+ := \max(w,0)$). 
\par 
To see that both $\widetilde{u}$ and $\widetilde{v}$ are well-defined and nonnegative, we use the dual setting of Lemma~\ref{lem:dua} of Appendix A: here $\mu_1(u,v):=d_1+\rho_\eta\star [p_\delta(u^+,v^+)]$ and $\mu_2(u,v):=d_2+\rho_\eta\star [q_\delta(u^+,v^+)]$ are both bounded and possess a strictly positive lower bound $\min(d_1,d_2)$.
\par
 Even more : thanks to the convolution and truncation operator, we also know that
 $\mu_k(u,v)\in\L^\infty(\R_+;\textnormal{W}^{2,\infty}(\T^d))$ for $k=1,2$, with $\|\Delta \mu_k(u,v)\|_{\L^\infty(\R_+;\L^\infty(\T^d))} \leq a_\delta$. We thus infer from estimate \eqref{ineq:dua:bis} of Lemma~\ref{lem:dua} of Appendix A that for every $(u,v)\in E$, the corresponding solutions $\widetilde{u},\widetilde{v}$ satisfy for $t\in\R_+$ the bound:
    \begin{align}\label{ineq:dua:fix}
\|(\widetilde{u},\widetilde{v})(t, \cdot)\|_2^2 + \int_0^t\int_{\T^d} d_1 |\nabla \widetilde{u}|^2 + \int_0^t\int_{\T^d} d_2 |\nabla \widetilde{v}|^2 \leq e^{a_\delta t}\, \|(\uin,\vin)\|_{\L^2(\T^d)}^2 .
    \end{align}
    In particular $(\widetilde{u},\widetilde{v})$ belongs to our weighted space $\E\times \E$, and we can serenely question the existence of a fixed point for the map $\Theta:\E^2\ni (u,v)\mapsto (\widetilde{u},\widetilde{v}) \in \E^2$.
\medskip

 One can readily check that $\Theta(\E\times \E)$ is relatively compact in $\E\times\E$. Indeed, thanks to \eqref{ineq:dua:fix}, we have a uniform control on the spatial gradients, and (thanks to the equations) we also have a control for the time derivatives in (at least) $\L^\infty(\R_+;\H^{-2}(\T^d))$. This is sufficient to invoke the Aubin-Lions lemma and recover a.e. convergence. Convergence in the $\E\times\E$ topology is then obtained thanks to the bound in $\L^\infty(\R_+;\L^2(\T^d))$ obtained through \eqref{ineq:dua:fix} and Vitali's lemma (the behavior for large times being handled thanks to the chosen decaying weight).

    \medskip

    Now that relative compactness of $\Theta(\E\times\E)$ has been established, we recall (see Lemma~\ref{lem:dua} of Appendix A) that the equations defining $\widetilde{u}$ and $\widetilde{v}$ are uniquely solvable in $\Ll^2(\R_+;\L^2(\T^d))$. In particular, if $(u_n,v_n)_n$ converges in $\E\times\E$ towards $(u,v)$, the sequence $(\Theta(u_n,v_n))_n$ is bounded in $\Ll^2(\R_+;\L^2(\T^d))$,  so that it has a unique possible cluster point in $\E\times\E$, which has to be $(\widetilde{u},\widetilde{v})$: continuity of $\Theta$ is established and Schauder's fixed point Theorem ensures the existence of a solution $(u,v)\in\E^2$ to \eqref{eq:syst_gen_reg}. 
Note finally that point 3. of Lemma~\ref{lem:dua} of Appendix A ensures that this fixed point is nonnegative, so that, in \eqref{eq:syst_gen_reg}, $p_\delta(u^+,v^+)$ and $q_\delta(u^+,v^+)$ can respectively be replaced by $p_\delta(u,v)$ and $q_\delta(u,v)$.

    \medskip
    
    Up to now (in the proof), $\delta$ was any strictly positive real number. We now take for $\delta$ the strictly positive real number obtained in the statement of the theorem (and we also restrict ourselves to the initial data
which are considered in the theorem). 
Then, we show that the solution $(u,v)$ of system (\ref{eq:syst_gen_reg}) obtained by the fixed point process actually solves a system in which the truncations are removed, that is
 \begin{equation}
  \label{eq:syst_gen_reg_bis}
  \left\{
    \begin{lgathered}
\partial_t u-\Delta\big[(d_1+ \rho_\eta \star p(u,v))u]=0,\\
      \partial_t v-\Delta\big[(d_2+ \rho_\eta \star q(u,v))v]=0.                                                                                                                                     \end{lgathered}
  \right.
\end{equation}
In order to do so, we need to show the bound $\sup_{t\geq 0}\|(u,v)\|_{\L^\infty([0,t]\times\T^d)} < \delta$. 
Let's first note that at this level, both $u$ and $v$ are smooth. Indeed, for fixed $\eta$, the mobilities $\mu_1 = d_1+ \rho_\eta \star p_\delta(u,v)$ and $\mu_2(u,v) = d_2 + \rho_\eta\star q_\delta(u,v)$ both belong to $\cap_{k\in\N} \L^\infty(\R_+;\W^{k,\infty}(\T^d))$. This allows, for instance for $u$, to expand the diffusion operator into a classical one : $\partial_t u - \Delta(\mu_1 u) = \partial_t u - \nabla \cdot (\mu_1 \nabla u) - \nabla \mu_1 \cdot \nabla u - u \Delta \mu_1$, and infer from classical parabolic theory (initial data are also smooth) the smoothness of $u$ (and similarly, of $v$), with respect to both variables. Thanks to this smooth setting, the map $t\mapsto \|(u,v)\|_{\L^\infty([0,t]\times\T^d)}$ is continuous and the assumption on the initial data ensures that at time $t=0$, this function is strictly less than $\delta$.
 By continuity, 
the set  $\left\{t>0\,:\, \|(u,v)\|_{\L^\infty([0,t]\times\T^d)} < \delta \right\}$ is therefore not empty. 
Assuming that its supremum is a finite real number $T^\star$, we note that, for any $T<T^\star$, the $\delta$-truncation in our approximation is  useless so that, on $[0,T]\times\T^d$, our solution $(u,v)$ is indeed solution to \eqref{eq:syst_gen_reg_bis}.
Now, let's check that the extra convolutions do not interfere with the computations that we performed in the proof of Proposition~\ref{prop:uni}.
\medskip

We now define $\Phi_1:=(d_1+[\rho_\eta\star p(u,v)])u$ and $\Phi_2:=(d_2+[\rho_\eta\star q(u,v)])v$, and compute
    \begin{align*}
      \partial_t \Phi_1 &= (d_1+[\rho_\eta\star p(u,v)]) \partial_t u + u \rho_\eta \star (\partial_1 p(u,v)\partial_t u) +u \rho_\eta \star (\partial_2 p(u,v)\partial_t v),\\
      \partial_t \Phi_2 &= (d_2+[\rho_\eta\star q(u,v)]) \partial_t v + v \rho_\eta \star (\partial_1 q(u,v)\partial_t u) +v \rho_\eta \star (\partial_2 q(u,v)\partial_t v).
      \end{align*}
      We also define (for $k> 1+ \frac{d}2$ given in the statement of the theorem, and $T<T^\star$)
      \begin{align*}
  \lambda(T) := \|\Delta\Phi_1\|_{\L^k(Q_T)} + \|\Delta\Phi_2\|_{\L^k(Q_T)} + \|u\|_{\L^\infty(Q_T)} +\|v\|_{\L^\infty(Q_T)}.
      \end{align*}
Since all quantities are smooth, this function is continuous.
      We infer from the previous equalities the following estimate:
    \begin{align}\label{me}
\|\partial_t \Phi_i  - d_i \Delta \Phi_i\|_{\L^k(Q_T)} \leq P(\lambda(T)),
      \end{align}        
    for some polynomial $P\in\R_+[X]$ having a double $0$ at the origin and depending only on universal constants and the data of our problem. This last estimate is the analog of \eqref{ineq:phi} in the proof of Proposition~\ref{prop:uni}. The remaining part of the proof
 is identical to the one given in Proposition~\ref{prop:uni}.
 Lastly we use the usual continuity argument to conclude that $T^\star=+\infty$.
      
      \medskip

      Now that we obtained a solution $(u_\eta,v_\eta)$ to our regularized system \eqref{eq:syst_gen_reg_bis} which remains uniformly small globally in time, it remains to pass to the limit and get rid of the regularization parameter $\eta>0$. 
The solutions of our approximated system \eqref{eq:syst_gen_reg_bis} are regular enough to justify the usual energy estimate obtained by multiplying the first equation by $u_\eta$, and the second by $v_\eta$. For the first equation, we get after usual integration by parts the estimate
      \begin{multline*}
\frac 12 \frac{\dd}{\dd t } \|u_\eta(t)\|_2^2 + \int_{\T^d} (d_1+\rho_\eta \star p(u_\eta,v_\eta))|\nabla u_\eta|^2 \\+ \int_{\T^d} u_\eta \nabla u_\eta \cdot \rho_\eta \star (\partial_1 p(u_\eta,v_\eta)\nabla u_\eta) \\+  \int_{\T^d} u_\eta\nabla u_\eta \cdot \rho_\eta \star (\partial_2 p(u_\eta,v_\eta)\nabla v_\eta)= 0.
\end{multline*}
In particular, since $p\in\R_+[X]$, we infer from Cauchy-Schwarz's inequality that
      \begin{multline*}
        \frac12 \frac{\dd}{\dd t} \|u_\eta(t)\|_2^2 + \frac{d_1}{2}\|\nabla u_\eta(t)\|_2^2  \\+ \|\nabla u_\eta(t)\|_2^2\left(\frac{d_1}{2}-\|u_\eta\|_\infty\partial_1 p(\|u_\eta\|_\infty,\|v_\eta\|_\infty)\right)\\- \|\nabla u_\eta(t)\|_{2} \|\nabla v_\eta(t)\|_{2} \|u_\eta\|_\infty  \partial_2 p(\|u_\eta\|_\infty,\|v_\eta\|_\infty)\leq 0. 
\end{multline*}

Summing with the analogous estimate for the equation satisfied by $v_\eta$, we recover the estimate
\begin{multline*}
  \frac12 \frac{\dd}{\dd t}\Big( \|u_\eta(t)\|_2^2+\|v_\eta(t)\|_2^2\Big) + \frac{d_1}{2}\|\nabla u_\eta(t)\|_2^2 + \frac{d_2}{2}\|\nabla v_\eta(t)\|_2^2 \\+ Q(\|u_\eta\|_\infty,\|v_\eta\|_\infty)(\|\nabla u_\eta(t)\|_2,\|\nabla v_\eta(t)\|_2) \leq 0,
\end{multline*}
where $Q(a,b)$ is the bilinear form with matrix
\begin{align*}
  \begin{pmatrix}
   \frac12 d_1-|a\partial_1 p(a,b)| & -|a\partial_2 p(a,b)|\\
 -|b\partial_1 q(a,b)| &     \frac12 d_2-|b\partial_2 q(a,b)| 
   \end{pmatrix}.
\end{align*}
Since $(a,b)\mapsto Q(a,b)$ is continuous and diagonal positive at $(0,0)$, we infer the existence of $\delta_A>0$ such that $|a|\vee |b| <\delta_A \Longrightarrow Q^{sym}(a,b)>0$ in the sense of symmetric matrices. In particular, since $|u_\eta|\vee|v_\eta|<\delta$, if we assume from the beginning that $\delta\leq\delta_A$,
 we have that $Q^{sym}(\|u_\eta\|_\infty,\|v_\eta\|_\infty)\geq 0$. At the end of the day, we get the estimate
\begin{align}\label{ineq:nrj:approx}
  \frac12 \frac{\dd}{\dd t}\Big( \|u_\eta(t)\|_2^2+\|v_\eta(t)\|_2^2\Big) + \frac{d_1}{2}\|\nabla u_\eta(t)\|_2^2 + \frac{d_2}{2}\|\nabla v_\eta(t)\|_2^2 \leq 0.
\end{align}
Note that in the above argument, the property of the map $A: (u,v) \mapsto ((d_1+p(u,v))\,u, (d_2+q(u,v))\,v)$ which is key is the fact that $D(A)^{\rm{sym}}(0,0) >0$ in the sense of symmetric matrices.
\par 
This gives enough control on the gradients to invoke Aubin-Lions' lemma and recover a.e. convergence for $(u_\eta,v_\eta)_\eta$. Since we also have a uniform bound in $\L^\infty(\R_+;\L^\infty(\T^d))$, this leads eventually to a global solution as announced.
\medskip

 Note that the stated large time behavior in \eqref{ineq:large} can also be recovered by \eqref{ineq:nrj:approx}, just as in the proof of \cite[Theorem 2.1]{BLMP}. Indeed, from \eqref{ineq:nrj:approx} and the Poincaré-Wirtinger inequality, we infer
\begin{align*}
  \frac12 \frac{\dd}{\dd t} \Big( \|u_\eta(t)\|_2^2+\|v_\eta(t)\|_2^2\Big) \lesssim - \|u_\eta(t)-\langle u_\eta(t)\rangle_{\T^d}\|_2^2 - \|v_\eta(t)-\langle v_\eta(t)\rangle_{\T^d}\|_2^2,
\end{align*}
where $\langle \cdot\rangle_{\T^d}$ denotes here the average on $\T^d$. The system being conservative, and using $\|u_\eta(t)-\langle u_\eta(0)\rangle_{\T^d}\|_2^2 = \|u_\eta(t)\|_2^2 - \langle u_\eta(0)\rangle_{\T^d}^2$ (and the same identity for $v_\eta$), we have $\Theta_\eta'(t) \lesssim - \Theta_\eta(t)$ (with a constant independent of $\eta$), where
 \[\Theta_\eta(t) := \|u_\eta(t)-\langle u_\eta(0)\rangle_{\T^d}\|_2^2+\|v_\eta(t)-\langle v_\eta(0)\rangle_{\T^d}\|_2^2.\]
The (uniform in $\eta$) exponential decay of $\Theta_\eta$ is then kept as $\eta\rightarrow 0$.
 $\qedhere$    
  \end{proof}

\section{Proof of the stability and uniqueness theorem}\label{subsec:locstab}

In this section, we prove Theorem \ref{uniqueness}.
This result is reminiscent of the stability estimate derived in \cite[Theorem 1]{bmmh} and is more precisely a consequence of Lemma~\ref{lem:dua:bis}. 
This stability estimate applies for bounded solutions of the system if one of them is small enough. It is a consequence of the more precise Proposition \ref{pr2} below. In order to write it down, we 
 introduce for $R>0$ the cone $\C^+_R$ containing all nonnegative elements of the $\L^\infty(\R_+;\T^d)^2$ ball of size $R$,  and $\L_R$ a common Lipschitz constant for both $p$ and $q$ on the compact set $\{(u,v)\in\R^2\,:\, \max(|u|,|v|)\leq R\}$.

  \begin{prop} \label{pr2}
Let $d\in \N^*$,  $d_1,d_2>0$, and $p,q$ be two elements of $\R_+[X,Y]$ vanishing at $(0,0)$. 
For any $R, \delta>0$ 
 such that \begin{align}\label{ineq:delta}
(\delta \L_R)^2\left[\frac{1}{d_1}+\frac{1}{d_2}\right] < \max(d_1 , d_2) ,
    \end{align}
 and any pair of $\mathscr{C}^0(\R_+;\H^{-1}(\T^d))$ of dual solutions $(U_1,U_2)\in\C_R^+\times\C_\delta^+$ (with $U_k=(u_k,v_k)$) to system \eqref{eq:syst_gen} in the sense of Definition~\ref{def:dua}, associated to initial data $(U_{1,in},U_{2,in}) \in \L^{\infty}(\T^d) \cap \H^{-1}(\T^d)$, one has
    \begin{multline*}
\|U_1(t)-U_2(t)\|_{\H^{-1}(\T^d)}^2 + \C\, \int_0^t \|U_1(s)-U_2(s)\|_{\L^2(\T^d)}^2\,\dd s  \\\leq \|U_{1}(0) -U_{2}(0)\|_{\H^{-1}(\T^d)}^2 +\left(\int_{\T^d} [u_{1}(0)-u_{2}(0)] \right)^2  \,T\, (d_1+p(R,R))\\  + \left(\int_{\T^d} [v_{1}(0)-v_{2}(0)] \right)^2 \, T\, (d_2+q(R,R)),
\end{multline*}
where the constant $\C$ depends only on $\delta, R,d_1,d_2$, $p$ and $q$.
  \end{prop}

  \begin{proof}
    The definition of $L_R$ ensures that 
 for any pair $(u_k,v_k)$, $k=1,2$,  one has  
    \[ \max\Big( |p(u_1,v_1)-p(u_2,v_2)|^2 ,  |q(u_1,v_1)-q(u_2,v_2)|^2 \Big) \leq \L_R^2\, \big( |u_1-u_2|^2+|v_1-v_2|^2\big)\]
on the compact set $\{(u,v)\in\R^2\,:\, \max(|u|,|v|)\leq R\}$.
\medskip
    
Let's write 
 $z_u:=u_1-u_2$ and $z_v:=v_1-v_2$. One has
    \begin{align*}
      \partial_t z_u - \Delta\big[(d_1+p(u_1,v_1))z_u\big] &= \Delta\big[(p(u_1,v_1)-p(u_2,v_2))u_2\big],\\
      \partial_t z_v - \Delta\big[(d_2+q(u_1,v_1))z_v\big] &= \Delta\big[(q(u_1,v_1)-q(u_2,v_2))v_2\big].
    \end{align*}
    Now $u_1$ and $v_1$ are nonnegative and bounded, so that $\mu_p:=d_1+p(u_1,v_1)$ and $\mu_q:=d_2+q(u_1,v_1)$ are both bounded from above and below by positive constants. We can therefore use Lemma~\ref{lem:dua:bis}, firstly for $z_u$, to get the estimate 
    \begin{multline*}
 \|z_u(t, \cdot)\|_{\H^{-1}(\T^d)}^2 + \int_0^t \int_{\T^d} \mu_p \, z_u^2\\ \leq \|z_u(0)\|_{\H^{-1}(\T^d)}^2 + \left(\int_{\T^d} z_u(0)\right)^2 \int_0^t \int_{\T^d} \mu_p + \int_{0}^t \int_{\T^d} \frac{\displaystyle\Big[(p(u_1,v_1)-p(u_2,v_2))u_2\Big]^2}{\mu_p}.
\end{multline*}
Using the definition of the common Lipschitz constant $\L_R$ and the fact that $p\in\R_+[X]$, with $U_1\in\C_R^+$ and $|u_2| \le \delta$, we have for $t\leq T$,
    \begin{multline*}
 \|z_u(t, \cdot)\|_{\H^{-1}(\T^d)}^2 + \int_0^t \int_{\T^d} d_1 \, z_u^2\\ \leq \|z_u(0)\|_{\H^{-1}(\T^d)}^2 + \left(\int_{\T^d} z_u(0)\right)^2  T (d_1+p(R,R)) + \frac{\L_R^2 \delta^2}{d_1} \int_{0}^t \int_{\T^d}(z_u^2+z_v^2).
\end{multline*}
A similar computation leads to an analogous estimate for $z_v$:
    \begin{multline*}
 \|z_v(t, \cdot)\|_{\H^{-1}(\T^d)}^2 + \int_0^t \int_{\T^d} d_2 \, z_v^2\\ \leq \|z_v(0)\|_{\H^{-1}(\T^d)}^2 + \left(\int_{\T^d} z_v(0)\right)^2  T (d_2+q(R,R)) + \frac{\L_R^2 \delta^2}{d_2} \int_{0}^t \int_{\T^d}(z_u^2+z_v^2).
\end{multline*}
The smallness assumption \eqref{ineq:delta} on $\delta$ ensures that
\begin{align*}
\C_\delta :=  \inf\bigg[ \left(d_1-\frac{\L_R^2\delta^2}{d_1}-\frac{\L^2_R\delta^2}{d_2}\right) , \left(d_2-\frac{\L_R^2\delta^2}{d_1}-\frac{\L^2_R\delta^2}{d_2}\right)\bigg]  >0. 
\end{align*}
Summing the estimates on $z_u(t, \cdot)$ and $z_v(t, \cdot)$, we infer 
\begin{multline*}
\|(z_u,z_v)(t, \cdot)\|_{\H^{-1}(\T^d)}^2 + \C_\delta \int_0^t \int_{\T^d} (z_u^2+z_v^2)  \\\leq \|(z_u(0),z_v(0))\|_{\H^{-1}(\T^d)}^2 +\left(\int_{\T^d} z_u(0)\right)^2  T (d_1+p(R,R))\\  + \left(\int_{\T^d} z_v(0)\right)^2  T (d_2+q(R,R)),
\end{multline*}
which concludes the proof. $\qedhere$
\end{proof}

 \section{Some considerations on the case when the initial data are not small}\label{sec-cc}

In this short section, we try to explain some of the difficulties for showing existence of weak solutions to general systems of the form (\ref{eq:syst_gen}) (when $d_1,d_2>0$ are two constants, $p,q$ are  two elements of $\R_+[X,Y]$  vanishing at $(0,0)$, and when periodic boundary conditions are imposed).
\medskip

Let us first notice that when no Lyapunov functional is known for the system (apart from obvious ones, such as the $\L^1$ norm of the two equations), no source of strong compactness for sequences of (approximated) solutions seems to appear (at least without further assumptions on $p,q$), since the usual \emph{a priori} estimates performed on classical (non cross) diffusion systems do not seem to hold. 
\medskip

Let us focus on a more favorable case, that is when some Lyapunov functional (also called energy in this context) is known to exist. When it is convex, we recall that the existence of weak solutions can often be obtained \cite{CJ,DLM,DLMT,LM}. However, there are cases when an energy exists,
 but is not convex. Such a situation has indeed been pointed out in \cite{BLMP}, where the following example is presented:
\begin{equation}
  \label{eq:syst_nonconv}
  \left\{
    \begin{lgathered}
\partial_t u - \Delta\big[(1 + v^2)u\big] = 0,\\
\partial_t v - \Delta\big[(1 + u^2)v\big] = 0.
    \end{lgathered}
  \right.
\end{equation}
One checks indeed that the space integral of $(1+u^2)(1+v^2)$ (we consider here periodic boundary conditions) decays along time. 
More precisely, one can check that for any smooth solution, the following bound holds:

\begin{align}
\label{ent:nonconv}\frac12\frac{\dd}{\dd t} \int_{\T^d} (1+u^2)(1+v^2) = -\int_{\T^d} \Big[|\nabla(u(1+v^2))|^2 + |\nabla(v(1+u^2))|^2\Big] \leq 0.
\end{align}

This dissipation of the energy controls in this situation only \emph{a part} of the gradients of the solution. To highlight this lack of control of the whole gradients, 
we present a short proposition (in which we omit the time variable) which shows the type of pathologies which could prevent strong compactness.



\begin{prop}\label{prop:contrex}
  There exist two sequences of functions $(u_n)_n$ and $(v_n)_n$ taking positive values on $\T^d$ such that :
  \begin{itemize}
  \item[$\bullet$] $(u_n)_n$ and $(v_n)_n$ are bounded in $\L^\infty(\T^d)$ ;
  \item[$\bullet$] $(u_n(1+v_n^2))_n$ and $(v_n(1+u_n^2))_n$ are bounded in $\H^1(\T^d)$ ;
    \item[$\bullet$] neither $(u_n)_n$ nor $(v_n)_n$ is relatively compact in $\L^2(\T^d)$.
  \end{itemize}
\end{prop}

\begin{proof}
The polynomial 
\begin{align*}
  P(X) &:= (X^2-3X+1)(X^3+X-3) \\
  &=(X^2+1)^2X-3(X^2+1)^2+9X
\end{align*}
admits (at least) two distinct positive roots: $r_1<r_2$. Consider the function $\psi := \mathbf{1}_{(0,1/2)}-\mathbf{1}_{(1/2,1)}$. Its $1$-periodization $\Psi := \sum_{k\in\Z} \tau_k\psi$ is constant on each half-integer interval $(k,k+1)/2$ for $k\in\Z$, alternates values $1$ and $-1$ on those, so it has a vanishing integral on each integer interval $(k,k+1)$. Define for $n\in\N^\star$ the function $\Psi_n:x\mapsto \Psi(2^n x)$ which is now constant on each dyadic interval $(k,k+1)/2^{n+1}$ and alternates values $1$ and $-1$ on those. In particular, for $m<n$, $\Psi_n$ has a vanishing integral on each dyadic interval $(k,k+1)/2^{m+1}$. We claim that $(\Psi_n)_n$ (restricted to $(0,1)$) is an orthonormal family in $\L^2(0,1)$ : each function is certainly of norm $1$ (since its takes $\pm 1$ as value), and if $m<n$, we can (essentially) cover $(0,1)$ by a finite family $\mathscr{F}$ of dyadic intervals $(k,k+1)/2^{m+1}$, so that
\begin{align*}
\int_0^1 \Phi_n \Phi_m = \sum_{\textnormal{I}\in\mathscr{F}} \int_{\textnormal{I}} \Phi_n \Phi_m,
\end{align*}
and as noticed before, on each element of $\mathscr{F}$, the function $\Phi_m$ is constant (by definition),
 whereas $\Phi_n$ has a vanishing integral : we have $\Phi_n \perp \Phi_m$ in $\L^2(0,1)$.
 A similar construction can be done on $\T^d$, that we still denote $(h_n)_n$.
The sequence \[v_n:= r_2 \frac{1+h_n}{2} + r_1\frac{1-h_n}{2}\] takes therefore its values in $\{r_1,r_2\}$ (so that $P(v_n)=0$) and does not admit a subsequence converging almost everywhere. Indeed, for $n\neq p$, orthogonality leads to
\begin{align*}
\|v_n-v_p\|_2 = \frac{r_2-r_1}{2} \|h_n-h_p\|_2 = \frac{r_2-r_1}{\sqrt{2}},
  \end{align*}
  which refutes the $\L^2(\T^d)$ Cauchy criterion for (any subsequence of) $(v_n)_n$, whereas almost everywhere convergence would imply convergence in $\L^2(\T^d)$ by dominated convergence. Letting $u_n:=3/(1+v_n^2)$, the identity $P(v_n) =0$ proves that $(1+u_n^2)v_n =3$, so that the two sequences $(u_n(1+v_n^2))_n$ and $(v_n(1+u_n^2))_n$ are indeed bounded in $\H^1(\T^d)$ (because they are constant). $\qedhere$
\end{proof}

\medskip

The counterexample given in Proposition~\ref{prop:contrex} relies on functions $u_n$ and $v_n$ that are rather singular (staircase functions) and this construction cannot be reproduced in a smooth setting (because of the $\H^1(\T^d)$ bound). In the context of the parabolic PDEs that we focus on, it is natural to expect more regularity for $u_n$ and $v_n$. Here's another way to describe the situation at stake. If $\Phi_n:=u_n(1+v_n^2)$ and $\Psi_n := v_n(1+u_n^2)$, the assumptions of Proposition~\ref{prop:contrex} imply that both $(\Phi_n)_n$ and $(\Psi_n)_n$ are bounded in $\H^1(\T^d)$. A direct computation shows that the identities defining these sequences are equivalent to $v_n = (1+u_n^2)^{-1}\Psi_n$ and
\begin{align*}
u_n (1+u_n^2)^2+ u_n \Psi_n^2- (1+u_n^2)^2 \Phi_n = 0.
\end{align*}
The previous equality can be written as follows:
\begin{align*}
u_n^5 + \sum_{k=0}^4 a_{k,n} u_n^k =0,
\end{align*}
where $(a_{k,n})_n$ for $k=0,\dots,4$, are five sequences of functions that are all bounded in $\H^1(\T^d)$. This means that at each point $x$, $u_n(x)$ takes its values in the set of roots of a polynomial whose coefficients (as functions of $x$) are all bounded in $\H^1(\T^d)$. In a sense, the counterexample in Proposition~\ref{prop:contrex} reduces to jumps (on a vanishing scale) from one root to another one. If we add a regularity assumption on $u_n$, this type of counterexample is not reproducible, but other pathologies may appear, especially in the case when the polynomial has multiple roots. It is possible that for a low-degree constraint, the existence of explicit formulas allows for simplification in the analysis, but in the specific case of fifth degree polynomial that we consider here, we find this issue rather intriguing.

   \appendix

   \section{Duality estimates}

The forthcoming duality estimate is given within a functional setting allowing its rigorous use in the context of (very) weak solutions.  

\begin{lem}\label{lem:dua}
 Under the assumptions of Lemma~\ref{lem:dua:bis}, if $f=0$, $z_{in}\in\L^2(\T^d)$  and additionnaly $\mu\in\Ll^\infty(\R_+;\W^{2,\infty}(\T^d))$, then

 \begin{enumerate}
\item
$z\in\mathscr{C}^0(\R_+;\L^2(\T^d))$, and we have the extra estimate:
    \begin{multline}\label{ineq:dua:bis}
\|z(t, \cdot)\|_{\L^2(\T^d)}^2 +\int_0^t \int_{\T^d} \mu\, |\nabla z|^2 \\\leq \|z_{in}\|_{\L^2(\T^d)}^2\,\, \exp \left(\int_0^t \|(\Delta\mu)^+(s, \cdot)\|_{\L^\infty(\T^d)}\,\dd s \right).
    \end{multline}
 \item 
    If furthermore $z_{in}\in\H^1(\T^d)$, then $\partial_t z \in\Ll^2(\R_+;\L^2(\T^d))$.
 \item
 Finally, if $0\leq z_{in}\in\L^\infty(\T^d)$, then for any $t\geq 0$ and a.e. $(s,x) \in [0,t]\times\T^d$,
    \begin{align}\label{ineq:dua:max}
      0 \leq z(s,x) \leq \|z_{in}\|_{\L^\infty} \exp\left(\int_0^t \|(\Delta\mu)^+(\sigma)\|_{\L^\infty(\T^d)}\,\dd \sigma \right). 
      \end{align}
\end{enumerate}
  \end{lem}

  \begin{proof}
\par
 For the estimate \eqref{ineq:dua:bis}, it is done for instance in \cite[Corollary 19]{dm}. 
\par
The case in which $\mu\in\L^\infty(0,T;\textnormal{W}^{2,\infty}(\T^d))$ and $z_{in}\in\H^1(\T^d)$ is in fact a consequence of the standard parabolic theory, because the operator $-\Delta(\mu z)$ may then be expanded. More precisely, we have
    \begin{align*}
\partial_t z -\mu\Delta z - 2 \nabla z \cdot \nabla \mu - z \Delta \mu =0,
    \end{align*}
    and the announced regularity is recovered by multiplying this equation by $-\Delta z$ and performing the usual estimates (integrations by parts and use of Young's inequality).
\par 
 Lastly, the maximum principle stated and proved in \cite[Corollary 18]{dm} ensures the uniform estimate (\ref{ineq:dua:max}) in the case of bounded nonnegative initial data. 
$\qedhere$
   \end{proof}

   \section{Bootstrap lemma}

         \begin{lem}\label{lem:boot}
Fix $P\in\R_+[X]$ such that $P(0)=P'(0)=0$. Then there exists $\delta>0$ such that, for any continuous map $\lambda:\R_+\rightarrow\R_+$ \[\lambda \leq \delta/2 +P(\lambda)\Longrightarrow \lambda < \delta.\] 
\end{lem}

\begin{proof}
Simply choose $\delta$ such that $0<x\leq \delta \Rightarrow P(x) < x/2$. This definition together with the inequality satisfied by $\lambda$ ensure the set equality $\{\lambda \leq \delta\} = \{\lambda < \delta\}$. The previous set is therefore closed, open and non-empty, and so it equals $\R_+$. $\qedhere$
\end{proof}

\section{Estimates for the heat kernel}

In this section, we recall some classical results about the heat kernel (on the torus $\T^d$). We recall that we denote $Q_T := [0,T] \times \T^d$.

We start with the Maximal regularity result for the heat flow, here in the context of the torus:

\begin{thm}\label{thm:max}
  For any $1<k<+\infty$, $m>0$ and $\ffi\in\mathscr{C}^\infty(\R_+\times \T^d)$ vanishing at $t=0$, one has
  \begin{align*}
    m\| \Delta \ffi \|_{\L^k(Q_T)} \lesssim_{k,d} \|\partial_t \ffi -m \Delta\ffi\|_{\L^k(Q_T)}.
   \end{align*}
\end{thm}

We then write down a classical semigroup property for the heat kernel, once again in the context of the torus:

  \begin{thm}\label{lem:unibound}
        For any $k>1+\frac{d}{2}$, any $\Phi\in\mathscr{C}^\infty(\R_+\times\T^d)$, and any $m,T > 0$, one has
        \begin{align*}
\|\Phi\|_{\L^\infty(Q_T)} \lesssim_{m,k,d} \|\Phi(0)\|_\infty + \|\partial_t \Phi -m \Delta\Phi\|_{\L^k(Q_T)} + \sup_{s\in[0,T]} \left|\int_{\T^d} \Phi(s)\right|.
         \end{align*}
      \end{thm}

 Theorems \ref{thm:max} and \ref{lem:unibound} are classical results in the case of the whole space or in the case of a smooth domain with Dirichlet boundary conditions (with a slightly different statement). We briefly explain below how they can be obtained in the case considered here (flat $d$-dimensional torus).
\medskip

We start with Theorem \ref{thm:max}. By homogeneity, it is sufficient to treat the case $m=1$. 
The equation $\partial_t \ffi - \Delta \ffi = F$ translates, in the (space-time) Fourier variables, as $i\tau \widehat{\ffi} + |\xi|^2 \widehat{\ffi} = \widehat{F}$. In particular, the map $F\mapsto \Delta \ffi$ is therefore the operator defined by the multiplier $(\tau,\xi)\mapsto \frac{-|\xi|^2}{i\tau +|\xi|^2}$, which satisfies the assumptions of the Mikhlin-Hörmander theorem. For more details (written in the case of the whole space $\R^d$), we refer for example to \cite[Appendix D.5]{robinson}.
\medskip

For Theorem \ref{lem:unibound}, we give some more details. Since
we are going to use a dyadic decomposition of the Fourier modes, we define 
 for $m\in\N^\star$:
    \[\C_m:=\{k\in\mathbf{Z}^d\,:\, m \leq |k| < 2m\}, \qquad 
\V_m:=\textnormal{Span}\{x\mapsto e^{i k\cdot x}\,:\,k\in\C_m\}   . \]
We recall first the

    \begin{lem}[Bernstein]\label{lem:bern}
For $m\in\N^\star$, $k\in[1,+\infty)$, $f\in\V_m$, one has  $\|f\|_\infty \lesssim (2m)^{\frac{d}{k}}\|f\|_k$.
    \end{lem}
It can be obtained by estimating the $L^{k'}$ norm of the Dirichlet kernel (on $\T^d$) and using 
Young's inequality for convolutions.
\medskip

We then state the

    \begin{lem}\label{lem:decay}
      There exists $\textnormal{c}>0$ such that, for $p \in[1,\infty]$, any $m\in\N^\star$ and $f\in\V_m$, the following estimate  holds:
      \begin{align*}
\forall t\geq 0,\qquad \left\| e^{t\Delta} f\right\|_p \lesssim e^{-\textnormal{c} tm^2} \| f\|_p.
        \end{align*}
      \end{lem}

      \begin{proof}
        There exists $\ffi\in\mathscr{D}(\R^\star)$ such that for any $g\in\V_m$, \[g=\sum_{k\in\Z^d} c_k(g) \ffi(|k|/m) e^{i k\cdot x}.\]
        This is in particular true for $g=e^{t\Delta} f$ for $t\geq 0$, and we thus infer
\[e^{t/m^2\Delta f}=\sum_{k\in\Z^d} c_k(f)e^{-t |k|^2 /m^2} \ffi(|k|/m) e^{i k\cdot x} = f\star \widecheck{\dd_m \ffi_t},\]
        where $\ffi_t(x):= e^{-tx^2} \ffi(x)$, $\dd_m\ffi_t(x) = \ffi_t(x/m)$, and for any test function $\psi\in\mathscr{D}(\R)$,  
  $$ \widecheck{\psi}(x) := \sum_{k\in\Z^d} \psi(|k|)e^{ik\cdot x}, $$
so that here
 $$ \widecheck{\dd_m \ffi_t}(x) := \sum_{k\in\Z^d} \ffi(|k|/m) e^{-t|k|^2/m^2} e^{ik\cdot x}. $$
Thanks to Young's inequality for convolutions, we only need to prove
 $$ \|\widecheck{\dd_m \ffi_t}\|_1 \lesssim e^{-\textnormal{c} t}. $$
This can be checked using \cite[Lemma 1.1]{I},
 from which we can extract the following result:
  For any $\psi\in\mathscr{D}(\R)$ and $m\in\N^\star$,
  $\|\widecheck{\dd_m \psi}\|_1 \lesssim \|\psi''\|_\infty$.
\medskip

In the particular case of $\psi=\ffi_t$, since $\ffi$ vanishes away from $0$, one has $\|\ffi_t''\|_\infty\lesssim e^{-\textnormal{c}t}$ and the estimate follows. $\qedhere$
\end{proof}

We infer from the two previous estimates 
the following one, for the heat flow generated by a source term in $\V_m$.
\begin{lem}\label{lem:heatan}
  For $T>0$, $m\in\N^\star$,  $k\in[1,+\infty)$ and $f\in\L^k(Q_T)$ such that $f(t, \cdot)\in\C_m$ for all $t\in[0,T]$, we have the following estimate:
  \begin{align*}
    \left\| \int_0^t e^{(t-s)\Delta}f(s)\,\dd s \right\|_{\L^\infty(Q_T)} \lesssim_d m^{\frac{d}{k}-\frac{2}{k'}} \|f\|_{\L^k(Q_T)}.
   \end{align*}
 \end{lem}
 \begin{proof}
   For $t\in[0,T]$ and $x\in\T^d$ we have, using Bernstein's Lemma~\ref{lem:bern},
 \begin{align*}
 \left|\int_0^t e^{(t-s)\Delta} f(s,x)\,\dd s\right| \leq \int_0^t \| e^{(t-s)\Delta} f(s)\|_\infty\,\dd s \leq  (2m)^{\frac{d}{k}}\int_0^t \| e^{(t-s)\Delta} f(s)\|_k\,\dd s.
  \end{align*}
    Now using the decay of the heat flow given in Lemma~\ref{lem:decay}, we infer  
\begin{align*}
 \left|\int_0^t e^{(t-s)\Delta} f(s,x)\,\dd s\right| \leq \int_0^t \| e^{(t-s)\Delta} f(s)\|_\infty\,\dd s \leq  (2m)^{\frac{d}{k}}\int_0^t e^{-(t-s)4\textnormal{c}m^2} \| f(s)\|_k\,\dd s,
\end{align*}
and the estimate follows directly from Hölder's inequality. $\qedhere$
                                                                                \end{proof}
We now can prove the following

\begin{prop}\label{thm:heat:main}
  Fix $T>0$, $k>1+\frac{d}{2}$. For all $f\in\L^k(Q_T)$ having a vanishing mean at all times, one has
  \begin{align*}
 \left\|\int_0^t e^{(t-s)\Delta} f(s)\,\dd s\right\|_{\L^\infty(Q_T)}
    \lesssim_{d,k} \|f\|_{\L^k(Q_T)}. 
    \end{align*}
  \end{prop}
  Before starting the proof, we introduce a few notations related to the Littlewood-Paley decomposition for elements defined on the torus $\T^d$. Any $f\in\mathscr{D}'(\T^d)$ decomposes uniquely as
  \begin{align*}
f := \sum_{k\in\Z^d} c_k(f) e^{i k\cdot x}.
  \end{align*}
  For $j\in\N$, we introduce the corresponding \emph{dyadic block}
  \begin{align*}
    f_j := \sum_{k\in\C_{2^j}} c_k(f) e^{i k \cdot x}.
  \end{align*}
  If $f$ has zero mean (that is, $c_0(f)=0$), we have (at least in $\mathscr{D}'(\T^d)$)
  \begin{align*}
    f= \sum_{j\in\N} f_j,
\end{align*}
and this decomposition is orthogonal in $\L^2(\T^d)$ if $f$ belongs to this space.
We first state a  result about a norm (which is actually of Besov type)  which naturally appears when dealing with the previous decomposition:

\begin{lem}\label{prop:besov}
  For $k\in[2,+\infty]$ and $f\in\L^k(\T^d)$ having a vanishing mean, introducing the sequence $b_f:=(\|f_j\|_k)_{j\in\N}$, one has \[\| b_f
    \|_{\ell^k(\N)} \lesssim_k \|f\|_k.\]
\end{lem}

\begin{proof}
For $k=2$ there is in fact (up to some constant) equality: this is simply Parseval's Theorem together with the orthogonality of the dyadic blocks. For $k=+\infty$, the inequality holds thanks to the continuity of $f\mapsto f_j$ from $\L^\infty(\T^d)$ to itself, with a norm uniformly bounded w.r.t.$\,j\in\N$. All the intermediate cases follow then from Riesz-Thorin interpolation Theorem. $\qedhere$ 
\end{proof}

\begin{proof}[Proof of Propositon~\ref{thm:heat:main}]
We use the Littlewood-Paley decomposition on $f(s)$ at each time $s\in[0,T]$ to infer
  \begin{align*}
\int_0^t e^{(t-s)\Delta} f(s) \,\dd s = \sum_{j\in\N} \int_0^t e^{(t-s)\Delta} f_j(s)\,\dd s.
  \end{align*}
  Using the triangular inequality and the estimate of Lemma~\ref{lem:heatan} for the heat flow generated by a source supported on an annulus, we get
  \begin{align*}
\left\| \int_0^t e^{(t-s)\Delta} f(s)\,\dd s\right\|_{\L^\infty(Q_T)} \lesssim_d \sum_{j\in\N} (2^j)^{\frac{d}{k}-\frac{2}{k'}} \|f_j\|_{\L^k(Q_T)}. 
\end{align*}
The assumption $k>1+\frac{d}{2}$ is equivalent to $\frac{d}{k}-\frac{2}{k'}<0$, so that we get by Hölder's inequality the estimate
  \begin{align*}
\left\| \int_0^t e^{(t-s)\Delta} f(s)\,\dd s\right\|_{\L^\infty(Q_T)} \lesssim_{d,k} \left(\sum_{j\in\N} \|f_j\|_{\L^k(Q_T)}^k \right)^{1/k}, 
\end{align*}
and the conclusion eventually follows from Lemma \ref{prop:besov}. $\qedhere$
      \end{proof}

      We end the proof of Theorem \ref{lem:unibound} by using Proposition~\ref{thm:heat:main}.

      \begin{proof}[Proof of Theorem \ref{lem:unibound}]
Introducing $\widetilde{\Phi} := \Phi-e^{t\Delta}\Phi(0)$ and $f:=\partial_t \Phi-\Delta\Phi$, we have in particular 
\[\widetilde{\Phi} = \int_0^t e^{(t-s)\Delta}f(s)\,\dd s.\]  For a function $\ffi$ on the torus, we introduce the notation \[\underline{\ffi} := \ffi-\int_{\T^d} \ffi,\] with an obvious extension for functions depending also on the time variable. We therefore have
\begin{align*}
\underline{\widetilde{\Phi}} = \int_0^t e^{(t-s)\Delta} \underline{f(s)}\,\dd s. 
\end{align*}
Using Proposition~\ref{thm:heat:main}, we infer for $k>1+\frac{d}{2}$
\begin{align*}
\|\underline{\widetilde{\Phi}}\|_{\L^\infty(Q_T)} \lesssim_{k,d} \|\underline{f}\|_{\L^k(Q_T)} \lesssim \|f\|_{\L^k(Q_T)}.
\end{align*}
On the other hand, we have
\begin{align*}
  \underline{\Phi}
  = \underline{e^{t\Delta}\Phi(0)}+
  \underline{\widetilde{\Phi}},  
\end{align*}
and the maximum principle for the heat equation implies the bound
\begin{align*}
\|\underline{e^{t\Delta} \Phi(0)}\|_{\L^\infty(Q_T)} \leq 2 \|\Phi(0)\|_\infty,
\end{align*}
so that  the announced estimate follows. $\qedhere$
\end{proof}

\section{Classical SKT system: another method yielding global uniform \emph{a priori} estimates}

In this short part of the Appendix, we propose another method enabling to obtain a global uniform estimate for solutions to the SKT system, which uses standard Sobolev norms. This method could be adapted to the general case of systems like \eqref{eq:syst_gen}, but since it involves taking many derivatives, we have chosen to write it in the specific case when $p$ and $q$ are both linear, that is in the case of the classical SKT system (without reaction terms):
\begin{equation}\label{eq:syst_SKT}
  \left\{
    \begin{lgathered}
\partial_t u - \Delta\big[(d_1 + a_1 u +b_1 v)u\big] = 0,\\
\partial_t v - \Delta\big[(d_2 + a_2 u+ b_2 v)v\big] = 0,
    \end{lgathered}
  \right.
\end{equation}
where $a_i,b_i,d_i$ are nonnegative constants. In the sequel, for any function $f$ defined on $\T^d$, we use the following notation \[\underline{f}:= f-\int_{\T^d}f.\]

\begin{prop}
  For $k>d/2$, there exists $\delta>0$ depending on the coefficients $a_i,b_i,d_i$ and on $k,d$, such that any smooth nonnegative solution $(u,v)$ to \eqref{eq:syst_SKT} satisfying
  \begin{align*}
\|\underline{u}(0),\underline{v}(0)\|_{\H^{k}(\T^d)} \leq \delta/2,
  \end{align*}
  satisfies furthemore 
  \begin{align*}
\|\underline{u}(t),\underline{v}(t)\|_{\H^{k}(\T^d)} + \frac12 \int_0^t \|\underline{u}(s),\underline{v}(s)\|_{\H^{k+1}(\T^d)}^2 \,\dd s \leq \delta,
  \end{align*}
  for all positive time $t$. In particular, due to the Sobolev embedding $\H^{k}(\T^d)\hookrightarrow \L^\infty(\T^d)$, this entails a global uniform bound for the solution. 
\end{prop}

\begin{proof}
Fix $\alpha\in\N^d$ and denote $\partial^\alpha$ the corresponding (spatial) differential operator. Applying it to the first equation of \eqref{eq:syst_SKT}, we get
\begin{align*}
\partial_t \partial^\alpha u - \Delta \big[(d_1+ a_1 u+b_1v) \partial^\alpha u\big] = \Delta\left[\sum_{0 < \beta \leq \alpha} c_{\alpha,\beta} \partial^\beta (a_1 u+b_1 v) \partial^{\alpha-\beta} u\right], 
\end{align*}
where we used Leibniz formula in the r.h.s. (with the partial order $\beta_i\leq \alpha_i$, for all $i$ on $\N^d$). Using Lemma~\ref{lem:dua:bis} with $z= \partial^\alpha u$ and $\mu = d_1+au+bv$, for all $\alpha\in\N^d$ satisying $0 < |\alpha|\leq k+1$, and taking the sum over all those multi-indices, we get 
\begin{multline*}
  \|\underline{u}(t,\cdot)\|_{\H^{k}(\T^d)}^2 +\int_0^t \|\underline{u}(s,\cdot)\|_{\H^{k+1}(\T^d)}^2\,\dd s \\\leq \|\underline{u}(0,\cdot)\|_{\H^{k}(\T^d)}^2 
  +\frac{1}{d_1} \sum_{0 < |\alpha|\leq k+1} \int_0^t \int_{\T^d} \left[\sum_{0 < \beta \leq \alpha} c_{\alpha,\beta} \partial^\beta (a_1 u+b_1 v) \partial^{\alpha-\beta} u\right]^2 .
\end{multline*}
A similar estimate holds for $v$, and summing both, we get for $\underline{U}:=(\underline{u},\underline{v})$ (and for a symbol $\lesssim$ which depends only on the parameters $a_i,b_i,d_i$):
\begin{multline}\label{ineq:gns}
  \|\underline{U}(t,\cdot)\|_{\H^{k}(\T^d)}^2 -  \|\underline{U}(0,\cdot)\|_{\H^{k}(\T^d)}^2+\int_0^t \|\underline{U}(s,\cdot)\|_{\H^{k+1}(\T^d)}^2\,\dd s \\\lesssim  \sum_{0< |\alpha|\leq k+1}\sum_{\substack{\beta+\gamma = \alpha\\0<|\beta|}} \int_0^t \int_{\T^d} |\partial^\beta U|^2 \cdot |\partial^\gamma U|^2\\ 
\lesssim  \sum_{0< |\alpha|\leq k+1} \int_0^t\int_{\T^d} |\partial^\alpha U|^2 \cdot |U|^2\\ + \sum_{0< |\alpha|\leq k+1}\sum_{\substack{\beta+\gamma = \alpha\\0<|\beta|,|\gamma|<|\alpha|}} \int_0^t \|\partial^\beta \underline{U}(s)\|_{2p}^2 \|\partial^\gamma \underline{U}(s)\|_{2p'}^2 \,\dd s,
\end{multline}
for any pair of conjugate exponents $1\leq p,p'\leq \infty$. In the final r.h.s., all the terms of the first sum are bounded by
\begin{align*}
\sup_{s\in[0,t]}\|\underline{U}(s)\|_{\infty}^2\, \int_0^t \|\underline{U}(s)\|_{\H^{k+1}(\T^d)}^2\,\dd s.
\end{align*}
Our goal is to obtain a similar bound for the double sum appearing on the last line. In order to do so, we will use the Gagliardo-Nirenberg-Sobolev interpolation inequality. For $0<|\beta|< |\alpha|$, this estimate implies 
\begin{align*}
\|\partial^\beta f\|_{2\frac{|\alpha|}{|\beta|}} \lesssim_{d,|\alpha|,|\beta|}  \|\D^{|\alpha|} f\|_2^{\frac{|\beta|}{|\alpha|}} \|f\|_\infty^{1-\frac{|\beta|}{|\alpha|}}, 
\end{align*}
since 
\begin{align*}
\frac{|\beta|}{2|\alpha|} = \frac{|\beta|}{d} + \frac{|\beta|}{|\alpha|}\left(\frac12 -\frac{|\alpha|}{d}\right).
\end{align*}
Now, getting back to the last double sum of \eqref{ineq:gns}, if $\beta\neq 0$ is linked to $\gamma\neq 0$ through $\beta+\gamma=\alpha$, the conjugate of $p=\frac{|\alpha|}{|\beta|}$ is $p'=\frac{|\alpha|}{|\gamma|}$. We have therefore the following estimate for any $|\alpha|\leq k+1$:
\begin{align*}
  \int_0^t \|\partial^\beta \underline{U}(s)\|_{2p}^2 \|\partial^\gamma \underline{U}(s)\|_{2p'}^2 \,\dd s &\lesssim_{d,k} \int_0^t \|\D^{|\alpha|}\underline{U}(s)\|_2^2\|\underline{U}(s)\|_\infty^2\,\dd s \\
  &\leq \sup_{s\in[0,t]}\|\underline{U}(s)\|_\infty^2 \int_0^t \|\underline{U}(s)\|_{\H^{k+1}(\T^d)}^2\,\dd s.
\end{align*}
At the end of the day, we infer from estimate \eqref{ineq:gns} the following one:
\begin{multline*}
  \|\underline{U}(t,\cdot)\|_{\H^{k}(\T^d)}^2 -  \|\underline{U}(0,\cdot)\|_{\H^{k}(\T^d)}^2 +\int_0^t \|\underline{U}(s,\cdot)\|_{\H^{k+1}(\T^d)}^2\,\dd s \\
 \lesssim   \sup_{s\in[0,t]}\|\underline{U}(s)\|_{\infty}^2\, \int_0^t \|\underline{U}(s)\|_{\H^{k+1}(\T^d)}^2\,\dd s,
\end{multline*}
where this time $\lesssim$ depends on the coefficients $a_i,b_i,d_i$ and also on $k$ and $d$.
\medskip

Defining 
$$ A(t) :=   \|\underline{U}(t,\cdot)\|_{\H^{k}(\T^d)}^2  + \int_0^t \|\underline{U}(s,\cdot)\|_{\H^{k+1}(\T^d)}^2\,\dd s , $$
we see that we can write 
$$ A(t) - A(0) \lesssim A(t)^2, $$
and the proof can be concluded thanks to Lemma \ref{lem:boot} of Appendix B.
$\qedhere$

\end{proof}

\bibliographystyle{plain}
\bibliography{artal}

\end{document}